\documentclass[12pt]{amsart}
\usepackage[all]{xy}
\usepackage{amsmath,amssymb,dsfont,mathtools,enumitem}
\usepackage[pdftitle={K-regularity of locally convex algebras},
  pdfauthor={Hvedri Inassaridze},
  pdfsubject={Mathematics}
]{hyperref}
\usepackage[lite]{amsrefs}
\usepackage{microtype}
\usepackage{wasysym}

\newtheorem{thm}{Theorem}[section]
\newtheorem{cor}[thm]{Corollary}
\newtheorem{lem}[thm]{Lemma}
\newtheorem{prop}[thm]{Proposition}
\theoremstyle{definition}
\newtheorem{defn}[thm]{Definition}
\theoremstyle{remark}
\newtheorem{rem}[thm]{Remark}

\numberwithin{equation}{section}

\def\Coker{\operatorname{Coker}}

\begin{document}

\title{K-regularity of locally convex algebras}
\author{Hvedri Inassaridze}
\address{A.~Razmadze Mathematical Institute of Tbilisi State University, 6, Tamarashvili Str., Tbilisi 0177, Georgia}
\address{Tbilisi Centre for Mathematical Sciences, Chavchavadze Ave.~75, 3/35, Tbilisi 0168, Georgia}
\email{inassari@gmail.com}
\phone{mobile +995 599157836 }
\thanks{The author would like to thank Larry Schweitzer for useful comments. This research was supported by Volkswagen Foundation, grant 85989 and Shota Rustaveli National Science Foundation, grant DI/12/5-103/11}

\begin{abstract} The isomorphism of Karoubi-Villamayor $K$-groups with smooth $K$-groups for monoid algebras over quasi stable locally convex algebras is established and we prove that the Quillen $K$-groups are isomorphic to smooth $K$-groups for monoid algebras over quasi-stable Frechet algebras having a properly uniformly bounded approximate unit. Based on these results the $K$-regularity property for quasi-stable Frechet algebras having a properly uniformly bounded approximate unit is established.
\end{abstract}

\maketitle

\section{introduction}

The $K$-regularity is an important $K$-theoretical property of rings. It is closely related to the homotopy property of functors and to the Fundamental Theorem of algebraic $K$-theory.The investigation of $K$-regularity appears in Grothendieck works for regular rings and is treated in problems of algebraic $K$-theory by Bass \cite{Ba}.

The starting point is the Grothendieck-Serre classical theorem stating that a regular ring is $K_{0}$-regular. This result was further extended by Bass-Heller-Swan for the Whitehead-Bass algebraic K-functor $K_{1}$ and by Quillen for all algebraic K-functors $K_{n}$, $n\geq 1$ \cite{Qu}. The $K$-regularity property of rings was also introduced and studied by Gersten \cite{Ge} and Karoubi \cite{Ka}.

In this direction more recent results were obtained by Gubeladze for the general case when the rings of polynomials are replaced by monoid rings satisfying certain conditions and the relationship was established between the $K$-regularity of a monoid ring $R$ and its algebraic properties \cite{Gu}. Such relationship was also studied by Cortinas, Haesemeyer and Weibel \cite{CoHaWe} when $R$ is a commutative ring of finite type over a field of characteristic $0$ or $R$ is a non-reduced scheme with underlying space an elliptic curve over a field of characteristic $0$. For commutative $C*$-algebras the $K$-regularity was established by Rosenberg \cite{Ro}. Further results about K-regularity of operator algebras were obtained for stable $C^\ast$-algebras \cite{In} and for stable profinite $C^{\ast}$-algebras \cite{InKa}.

Our aim is to find a wide class of locally convex algebras for which the $K$-regularity property holds. For this purpose first we will investigate the relationship between Karoubi-Villamayor algebraic $K$-theory and topological $K$-theory in the category of locally convex complex algebras. Then the existing comparison of Quillen algebraic K-theory and topological K-theory \cite{SuWo, Wo, CoTh, Ro1, InKa1} will be extended to monoid algebras over locally convex algebras. To establish these results the topological invariants introduced in \cite{InKa1} and called smooth $K$-groups will be used. It should be noted that K-theories related to smooth K-groups have been treated in \cite{CoTh, Ti}.

\;

%%%%%%%%%%%%%%%%%%%%%%%%%%%%%%%%%%%%%%%%%%%%%%%%%%%%%%%%%%%%%%%%%%%

\section{Preliminaries}

In this section we recall some definitions and propositions given in \cite{InKa1} which will be used later.

By a locally convex algebra we mean an algebra over the field of complex numbers equipped with a Hausdorff complete locally convex topology and jointly continuous multiplication The category $\mathbf{ALC}$ of locally convex algebras is closed under the Grothendieck projective
tensor product \cite{Tr}.

%% Definition
\begin{defn}
Let
$$
0\rightarrow C \rightarrow  B \rightarrow A \rightarrow 0
$$
be a sequence of morphisms in the category $\mathbf{LC}$ of locally convex linear topological spaces and continuous linear maps. It will be said that this sequence is a proper exact sequence if $f$ is a homeomorphism of $C$ on $Imf$, $g$ is an open surjective map and $Imf=Kerg$. It will be said  proper split exact if it is proper exact and $g$ has a right inverse in $\mathbf{LC}$. A short exact sequence in the category $\mathbf{ALC}$ is said proper exact sequence if it is split exact sequence in the category $\mathbf{LC}$.
\end{defn}

%% Proposition
\begin{prop}
If
$$
0\rightarrow C \rightarrow B \rightarrow A \rightarrow 0
$$
is a proper exact sequence of locally convex algebras and $D$ is a locally convex algebra, then the sequence
$$
0\rightarrow D\widehat\otimes C\rightarrow D\widehat\otimes B \rightarrow
D\widehat\otimes A \rightarrow 0
$$
is a proper exact sequence.
\end{prop}

The smooth $K$-theory was introduced in \cite{InKa1} for locally convex algebras. The same definition is valid for arbitrary real or complex topological algebras and is completely similar. Namely, let $A^{\infty(I)}$ be the topological algebra of smooth maps from the unit interval $I$ to the topological algebra $A$. Any continuous homomorphism of topological algebras $\varphi: A\rightarrow A'$ induces a homomorphism of topological algebras $\varphi^{\infty(I)}:
A^{\infty(I)}\rightarrow A'^{\infty(I)}$. For any topological algebra $A$ consider the evaluation maps at $t=0$ and $t=1$
$$
A^{\infty(I)} \rightarrow A, i=0,1, \varepsilon_{0}(f)-f(0), \varepsilon_{1}=f(1).
$$
Denote by $\mathfrak{I}(A)$ the kernel of $\varepsilon_{0}$ and by
$\tau_{A}: \mathfrak{I}(A) \rightarrow A$ the restriction of
$\varepsilon_{1}$ on $\mathfrak{I}(A)$. There is a smooth
homomorphism $\delta_{A}:\mathfrak{I}(A) \rightarrow \mathfrak{I}^{2}(A)$ sending $f\in \mathfrak{I}(A)$ to $\delta_{A}(f)(s,t)=f(st)$. One gets the smooth path cotriple $\mathfrak{I}$ (for locally convex algebras see \cite{InKa1}) which induces the augmented simplicial group
$$
GL(\mathfrak{I}^{+}_{*}(A))=GL(\mathfrak{I}_{*}(A)) \rightarrow GL(A).
$$

%% Definition
\begin{defn}
For any topological algebra $A$ the smooth $K$-functors $K^{sm}_{n}, n\geq0$ are defined as follows
$$
K^{sm}_{n}(A)=\pi_{n-2}GL(\mathfrak{I}_{*}(A))
$$
for $n\geq 3$, $K^{sm}_{0}(A)=K_{0}(A)$ and for $n=1,2$ are defined by the exact sequence
$$
0\rightarrow K^{sm}_{2}(A)\rightarrow\pi_{0}(GL(\mathfrak{I}_{*}(A))\rightarrow GL(A)\rightarrow K^{sm}_{1}(A)\rightarrow0.
$$
\end{defn}

All defined smooth $K$-groups are abelian (the proof for the case $n=1,2$ is similar to the case of locally convex algebras \cite{InKa1}).

Let $\mathfrak{F}$  and $\mathfrak{P}$ be the free cotriple and the polynomial cotriple respectively on the category of associative rings and $\mathfrak{J}$ the continuous path cotriple on the category of topological algebras. Then one has natural morphisms
$$
\mathfrak{F}\overset{\alpha}\rightarrow\mathfrak{P}\overset{\delta}\rightarrow\mathfrak{I}\overset{\beta}\rightarrow\mathfrak{J}.
$$
For any topological algebra $A$ the homomorphism $\alpha_{A}:\mathfrak{F}(A) \rightarrow\mathfrak{P}(A)$ is given by
$\alpha_{A}(|a|)(ax)$, $|a|\in F_{A}$, where $F_{A}$ is the free algebra generated by $A$, $\delta_{A}(ax) (t\mapsto at), t\in I$, and one has the inclusion $\beta_{A}: \mathfrak{I}(A) \rightarrow\mathfrak{J}(A)$.

The topological $K$-groups $K^{top}_{n}(A)=\pi_{n-2}GL(\mathfrak{J}^{+}_{*}(A))$, $K^{top}_{0}(A)=K_{0}(A))$, were defined by Swan \cite{Sw} and the algebraic K-groups $kv_{n}(A)=\pi_{n-2}GL(\mathfrak{P}^{+}_{*}(A))$ by Karoubi and Villamayor \cite{KaVi} for $n\geq 1$. The morphisms $\alpha$, $\delta$ and $\beta$ induce respectively functorial homomorphisms
$$
K_{n}(A) \rightarrow kv_{n}(A)\rightarrow K^{sm}_{n}(A)\rightarrow K^{top}_{n}(A)
$$
for $n\geq1$, where $K_{*}(A)$ are Quillen's $K$-groups which are isomorphic to Swan's algebraic $K$-functors defined in \cite{Sw} similarly to the topological case the continuous path cotriple replaced by the free cotriple. These homomorphisms are surjective for $n=1$.

%% Definition
\begin{defn}
Two homomorphisms $f,g: A \rightarrow B$ of topological algebras are said smoothly homotopic if there exists a continuous homomorphism $h: A \rightarrow B^{\infty(I)}$ such that $\varepsilon_{0}h=\varepsilon_{1}h$ which is called the smooth homotopy between $f$ and $g$.
\end{defn}

Denote by $\mathbf{Gr}$ and $\mathbf{Ab}$ the category of groups and abelian groups respectively.

%% Definition
\begin{defn}
A functor $T:\mathbf{A}\rightarrow \mathbf{Gr}$ is called smooth homotopy functor if $T(f)=T(g)$ for smoothly homotopic $f$ and $g$.
\end{defn}

%% Proposition
\begin{prop}\label{Prop_2.6}
Let $\mathbf{B}$ be a full subcategory of the category $\mathbf{A}$ of topological algebras containing with any topological algebra $A$ the topological algebra $A^{\infty(I)}$. A functor $T:\mathbf{B}\rightarrow \mathbf{Gr}$ is a smooth homotopy functor if and only if the inclusion
 $i:A \rightarrow A^{\infty(I)}$ induces an isomorphism $T(i):T(A) \rightarrow T(A^{\infty(I)})$ for any topological algebra $A$ of the category $\mathbf{B}$.
\end{prop}

%% Proposition
\begin{prop}
The topological $K$-functors $K^{top}_{n}$ and the smooth $K$-functors $K^{sm}_{n}$ are smooth homotopy functors for $n\geq1$ on the category of topological algebras.
\end{prop}

%% Definition
\begin{defn}\label{Def_2.8}
Let $f: B\rightarrow A$ be a continuous injective algebra homomorphism for a Banach algebra $B$. If an approximate unit of the Frechet algebra $A$ is the image of a bounded approximate unit of $B$, then it is called properly uniformly bounded approximate unit of the Frechet algebra $A$.
\end{defn}

If $A$ is a m-convex Frechet algebra, then set $A_{b}$ of uniformly bounded elements in $A$ becomes a Banach algebra with respect to the norm $\parallel a \parallel = sup \parallel a \parallel_{n}$. Suppose $A$ has a bounded approximate unit in $A_{b}$. Then by considering the inclusion $A_{b}\hookrightarrow A$ and by Definition \ref{Def_2.8} this approximate unit is a properly uniformly bounded approximate unit and in this particular case it is called uniformly bounded approximate unit of $A$.

If $A$ is an arbitrary unital Frechet algebra with the sequence of determining seminorms $\|.\|_{n}$, then with respect to the inclusion $\mathbb Ce\hookrightarrow A$ the unit $e$ is a properly uniformly bounded approximate unit of $A$. Note that the unit $e$ may be unbounded in $A$, depending on the given sequence of determining seminorms.

The Frechet algebras considered in this article are not necessarily locally multiplicatively convex (m-convex). Finite projective tensor products of Frechet algebras with properly uniformly bounded approximate unit are again Frechet algebras with properly uniformly bounded approximate unit. The unit element of any unital Frechet algebra $A$ is a properly uniformly bounded approximate unit of $A$. Projective limits of countably many C*-algebras are Frechet algebras with properly uniformly bounded approximate unit. Many important examples of functional algebras occur in analysis that are not locally multiplicatively convex and are Frechet algebras with properly uniformly bounded approximate unit (see Examples 1 - 4, \cite{InKa1}). If $A$ is a Frechet algebra with properly uniformly bounded approximate unit, then so is the Frechet algebra $A^{\infty(I)}$.

The main and needed $K$-theoretical property of Frechet algebras with properly uniformly bounded approximate unit is expressed in the following assertion:

%% Theorem
\begin{thm}\label{th_2.9}
If $A$ is a Frechet algebra with properly uniformly bounded approximate unit, then it possesses the TF-property  and therefore the excision property
 in algebraic $K$-theory and the $H$-unitality property.
\end{thm}

\;
%%%%%%%%%%%%

\section{Smooth Karoubi Conjecture and K-regularity}

Besides the aforementioned assertions given in \cite{InKa1} for locally convex algebras we will use the following important property of smooth maps:

%% Lemma
\begin{lem}\label{lem_3.1}
There is an isomorphism
$$
A^{\infty(I)}\widehat\otimes B \approx (A\widehat\otimes B)^{\infty(I)}
$$
for any $A$ and $B$ locally convex algebras.
\end{lem}
\begin{proof}
A minor generalization of Theorem 44.1 \cite{Tr} shows that $C^\infty(I)  {\widehat \otimes}_\epsilon A \cong C^\infty(I, A)$, where ${\widehat \otimes}_\epsilon$  is the completed $\epsilon$-tensor product.  By the nuclearity of $C^\infty(I)$, Theorem 50.1 \cite{Tr} says that $C^\infty(I) {\widehat \otimes} A \cong C^\infty(I) {\widehat \otimes}_\epsilon A$. Finally one gets
$$
A^{\infty(I)}\widehat\otimes B \approx ((C)^{\infty(I)}\widehat\otimes A)\widehat\otimes B \approx (C)^{\infty(I)}\widehat\otimes (A\widehat\otimes B) \approx (A\widehat\otimes B)^{\infty(I)}.
$$
This completes the proof.
\end{proof}

%% Remark
\begin{rem}
(1) For $m$-convex locally convex algebras Lemma \ref{lem_3.1} is proved in \cite{Tr}. The proof for the general case is due to Larry Schweitzer.

\noindent (2) It should be noted that this property doesn't hold for continuous maps and arbitrary locally convex algebras even for Banach algebras. That is the reason why we have introduced smooth $K$-groups instead of topological $K$-groups for the investigation of Karoubi Conjecture about the isomorphism of algebraic and topological $K$-functors in the case of locally convex algebras which we call the smooth Karoubi Conjecture  (see \cite{InKa1}).

\noindent (3) This isomorphism also holds when $(-)^{\infty(I)}$ is replaced by $\mathfrak{I}(-)$.
\end{rem}

%% Theorem
\begin{thm}\label{th_3.3}
 Let $\mathbf{B}$ be a full subcategory of the category $\mathbf{A}$ of topological algebras containing with any topological algebra $A$ the topological  algebra $A^{\infty(I)}$.
\begin{enumerate}
\item[(1)] the functors $K_{1}$ and $K^{sm}_{1}$ are isomorphic on the category $\mathbf{B}$ if and only if  $K_{1}$ is a smooth homotopy functor on $\mathbf{B}$,
\item[(2)] the functors $kv_{1}$ and $K^{sm}_{1}$ are isomorphic on the category $\mathbf{B}$ if and only if  $kv_{1}$ is a smooth homotopy functor on $\mathbf{B}$.
\end{enumerate}
\end{thm}
\begin{proof}
\noindent (1) Let $K_{1}$  and $K^{sm}_{1}$ be isomorphic functors on $\mathbf{B}$ and consider for any topological algebra $A$ the following commutative diagram
$$
\xymatrix{
 K_{1}(A) \ar[d]\ar[r] & K_{1}(A^{\infty(I)})\ar[d]\\
 K^{sm}_{1}(A) \ar[r] & K^{sm}_{1}(A^{\infty(I)})\\
 }
$$
where the vertical homomorphisms and the bottom homomorphism are isomorphisms implying the isomorphism of the top homomorphism. It remains to apply  Proposition \ref{Prop_2.6} showing the smooth homotopy property of $K_{1}$.

Let $K_{1}$ be a smooth homotopy functor on $\mathbf{B}$. The surjection $\mathfrak{I}(A)\rightarrow A$ induces the following commutative diagram with exact rows and columns
$$
\xymatrix{
 0 \ar[d] & 0 \ar[d] \\
 E(J(A)) \ar[d]\ar[r] & E(A) \ar[d]\ar[r] & 0 \\
GL(J(A)) \ar[d]\ar[r] & GL(A) \ar[d]\ar[r] & K_1^{sm}(A) \ar[d]\ar[r] & 0\\
K_1(J(A)) \ar[d]\ar[r]^{\;K_1(\tau)} & K_1(A) \ar[d]\ar[r] & \Coker K_1(\tau) \ar[r] & 0\\
0 & 0
 }
$$
which implies the isomorphism $K^{sm}_{1}(A)\approx CokerK_{1}(\tau)$.

Now we will use the notion of smooth homotopization $h^{sm}T$ of any functor $T:\mathbf{A}\rightarrow \mathbf{Gr}$ defined by
$$
h^{sm}T(A)=Coker (T(A^{\infty(I)})\rightrightarrows T(A)).
$$
By using the exactness property of the functor $K_{1}$, it is easily shown that
$$
(h^{sm}K_{1})(A)\approx CokerK_{1}(\tau).
$$
Thus one gets the isomorphism  $K^{sm}_{1}(A)\approx h^{sm}K_{1}(A)$. It follows that if $K_{1}(A)\approx K_{1}(A^{\infty(I)})$ one obtains the required isomorphism.

\noindent (2) For the first part the proof is similar to the case 1). Let $k v_{1}$ be a smooth homotopy functor on $\mathbf{B}$. Since $K^{sm}_{1}$ is a smooth homotopy functor, the homomorphism $\delta*_{1}: kv_{1}(A) \rightarrow K^{sm}_{1}(A)$ induces a homomorphism
$h^{sm}\delta*_{1}:h^{sm}(kv_{1})(A)\rightarrow K^{sm}_{1}(A)$.

The commutative diagram
$$
\xymatrix{
 0 \ar[d] & 0 \ar[d] \\
GL(J(A)) \ar[d]\ar[r] & GL(A) \ar[d]\ar[r] & K_1^{sm}(A) \ar@{=}[d]\ar[r] & 0\\
h^{sm}(kv_{1})(J(A)) \ar[d]\ar[r] & h^{sm}(kv_{1})(A) \ar[d]\ar[r] & K_{1}^{sm}(A) \ar[r] & 0\\
0 & 0
 }
$$
with exact top row and vertical surjective homomorphisms implies the exactness of the bottom row. The topological algebra $\mathfrak{I}(A)$ is contractible because of the trivial map $0_{\mathfrak{I}(A)}$ and the identity map $1_{\mathfrak{I}(A)}$ are smoothly homotopic with smooth homotopy $\delta_{A}:\mathfrak{I}(A)\rightarrow \mathfrak{I}^{2}(A)$ between them. Since $h^{sm}(kv_{1})$ is a smooth homotopy functor, one has $h^{sm}(kv_{1})(\mathfrak{I}(A)) = 0$ implying the isomorphism $h^{sm}(kv_{1})(A)\rightarrow K^{sm}_{1}(A)$. Thus $kv_{1}(A)\rightarrow K^{sm}_{1}(A)$ is an isomorphism if $kv_{1}(A)\rightarrow kv_{1}(A^{\infty(I)})$ is an isomorphism. This completes the proof.
\end{proof}

Let $M$ be a monoid and $A[M]$ a monoid algebra over a locally convex algebra $A$. Each determining seminorm $\mu$ of $A$ induces in a natural way a seminorm $\mu_{M}$ on $A[M]$ as follows
$$
\mu_{M}(\sum a_{j}m_{j}) = \sum \mu(a_{j})
$$
for any element $\sum a_{j}m_{j}$ of $A[M]$. Then $A[M]$ becomes a topological algebra with respect to these induced seminorms. Regarding the topology on the monoid algebra $A[M]$ over a locally convex algebra $A$,  it can be considered as the union of the set $S$ of finite products of copies of $A$ indexed by finite subsets of $M$ and partially ordered by inclusion. Then we take on $A[M]$ the union topology induced by the topology of these finite products.

To confirm Karoubi's conjecture about the isomorphism of algebraic and topological $K$-functors we will use the following important notion of rings that is called the triple factorization property introduced in \cite{SuWo}. We recall its definition.

%% Definition
\begin{defn}
It is said that a ring $A$ possesses the property $(TF)_{right}$ if for any finite collection of elements $a_{1},a_{2}, \ldots, a_{m}$ of the ring $A$ there exist elements $b_{1},b_{2}, \ldots, b_{m}, c, d \in A$ such that $a_{i}=b_{i}cd$ for $1\leq i\leq m$ and the left annihilators in $A$ of $c$ and $cd$ are equal.
\end{defn}

%% Proposition
\begin{prop}\label{prop_3.5}
If a ring $A$ has the $(TF)_{right}$ property, then the monoid algebra $A[M]$ has also the $(TF)_{right}$ property for any monoid $M$.
\end{prop}
\begin{proof}
Let
$$
\begin{matrix}
p_{1} = a_{11}m_{01} + a_{12}m_{11} + \cdots + a_{1n_{1}}m_{1n_{1}}\\
p_{2} = a_{21}m_{21} + a_{22}m_{22} + \cdots + a_{2n_{2}}m_{2n_{2}}\\
.............................................................\\
p_{k} = a_{k1}m_{k2} + a_{k2}m_{k2} + \cdots + a_{kn_{k}}m_{kn_{k}}
\end{matrix}
$$
be elements of the monoid algebra $A[M]$.The $(TF)_{right}$ property of the ring $A$ implies that for the elements
$$
a_{11}, \ldots , a_{1n_{1}}, a_{21}, \ldots, a_{2n_{2}}, \ldots, a_{k1}, \ldots, a_{kn_{k}}
$$
of the ring $A$ there exist elements
$$
b_{11},  \ldots, b_{1n_{1}}, b_{22}, \ldots, b_{2n_{2}}, \ldots, b_{k1}, \ldots, b_{kn_{k}}
$$
and $c,d$ of $A$ such that
$$
\begin{matrix}
a_{1j}=b_{1j}cd,  j=1, \ldots , n_{1}\\
a_{2j}=b_{2j}cd,  j=1, \ldots , n_{2}\\
......................................\\
a_{kj}=b_{kj}cd,  j=1, \ldots , n_{k}
\end{matrix}
$$
and if $ycd=0$ then $yc=0$ for $y\in A$.

Now consider the following elements of the monoid algebra $A[M]$:
$$
\begin{matrix}
q_{1} = b_{11}m_{01} + b_{12}m_{11} + \cdots + b_{1n_{1}}m_{1n_{1}}\\
q_{2} = b_{21}m_{21} + b_{22}m_{22} + \cdots + b_{2n_{2}}m_{2n_{2}}\\
............................................................\\
q_{k} = b_{k1}m_{k2} + b_{k2}m_{k2} + \cdots + b_{kn_{k}}m_{kn_{k}}
\end{matrix}
$$

It is clear that one has $p_{i} = q_{i}cd$ for $1\leq i\leq k$. If $pcd=0$ for some monoid algebra
$p = a_{1}m_{1} + a_{2}m_{2} + \cdots + a_{n}m_{n}$, then $a_{j}cd = 0$ for $0\leq j\leq n$. Thus one gets $a_{j}c = 0,  0\leq j\leq n $, implying $pc = 0$. Therefore the monoid algebra $A[M]$ has the $(TF)_{right}$ property.
\end{proof}

\;

This proposition generalizes Lemma 16 \cite{In}. Therefore the polynomial algebra $A[x_{1},x_{2}, \ldots, x_{m}]$, $n\geq 1$, and the Laurent polynomial algebra $A[t,t_{1}]$ over a Frechet algebra $A$ with properly uniformly bounded approximate unit possess the excision property in algebraic $K$-theory and the $H$-unitality property.

In what follows the following property of smooth maps will be used
\begin{equation}\label{Nom1}
(A[M])^{\infty(I)} \approx (A^{\infty(I)})[M]
\end{equation}
for any locally convex algebra algebra $A$. It is clear that this condition implies the isomorphism $\mathfrak{I}(A[M]) \approx (\mathfrak{I}(A))[M]$ too.

Let $M$ be a monoid and denote by $\mathbf{ALC}[M]$ the category of monoid algebras $A[M]$ over locally convex algebras $A$ and by $\mathds{C*}$ the category of $C*$-algebras. Let $T$ be an arbitrary functor T from the category $\mathbf{ALC}[M]$ to the category $\mathbf{Ab}$ of abelian groups.

%% Theorem
\begin{thm}\label{th_3.6}
If the functor
$$
T((A\widehat\otimes(-\otimes\mathcal{K}))[M]): \mathds{C*} \rightarrow \mathbf{Ab}
$$
is a stable and split exact functor, then the functor
$$
T((-\widehat\otimes\mathcal{K})[M]): \mathbf{ALC} \rightarrow \mathbf{Ab}
$$
is a smooth homotopy functor.
\end{thm}
\begin{proof}
The Higson homotopy invariance theorem \cite{Hi}, which is true for complex (or real) $C*$-algebras, implies that the functor
$$
T((A\widehat\otimes(-\bigotimes\mathcal{K}))[M])):\mathds{C*} \rightarrow \mathbf{Ab}
$$
is homotopy invariant.

By using Lemma \ref{lem_3.1} one has the following commutative diagram
$$
\xymatrix{
T\Big(\big((A^{\infty(I)}\widehat\otimes k)\widehat\otimes K \big) [M] \Big) \ar[d]_{\approx}\ar@<0.4ex>[r]\ar@<-0.4mm>[r] & T\Big(\big((A \widehat\otimes k)\widehat\otimes K \big) [M] \Big) \ar[d]^{\approx}\\
T\Big(\big(A\widehat\otimes (k^{\infty(I)} \widehat\otimes K) \big) [M] \Big) \ar[d]\ar@<0.4ex>[r]\ar@<-0.4mm>[r] & T\Big(\big(A\widehat\otimes (k\otimes K) \big) [M] \Big) \ar[d]\\
T\Big(\big(A\widehat\otimes(k^{I^n} \otimes K) \big) [M] \Big) \ar@<0.4ex>[r]\ar@<-0.4mm>[r] & T\Big(\big(A\widehat\otimes (k\otimes K) \big) [M] \Big)
}
$$

where horizontal maps of this diagram are induced by evaluation maps. The homotopy invariance of the functor $T((A\widehat\otimes(-\otimes\mathcal{K}))[M])$ implies the equality of the bottom two horizontal homomorphisms and therefore the equality of the top two horizontal homomorphisms.This shows the smooth homotopy invariance of the functor $T((-\widehat\otimes\mathcal{K})[M])$.This completes the proof.
\end{proof}

%% Definition
\begin{defn}
A locally convex algebra $B$ is called quasi-stable if it has the form $A {\widehat \otimes} \mathcal{K}$ for some locally convex algebra $A$. where $\mathcal{K}$ is the $C^{*}$-algebra of compact operators on the infinite dimensional Hilbert space $\mathcal{H}$.
\end{defn}

%% Theorem
\begin{thm}\label{th_3.8}
For any locally convex algebra $A$ there is an isomorphism
$$
kv_{n}((A\widehat\otimes\mathcal{K})[M])\rightarrow K^{sm}_{n}((A\widehat\otimes\mathcal{K})[M])
$$
for all $n\geq 1$.
\end{thm}
\begin{proof}
First it will be shown that $kv_{n}((-\widehat\otimes\mathcal{K})[M])$, $n\geq 1$, is a smooth homotopy functor on the category of locally convex algebras. According to Theorem \ref{th_3.6} it suffices to prove that the functor $kv_{n}((A\widehat \otimes(-\otimes \mathcal{K}))[M]$, $n\geq 1$, is stable and split exact on the category of C*-algebras for any locally convex algebra $A$.

Let
\begin{equation}\label{Nom3}
0\rightarrow C_{1}\rightarrow C\rightarrow C_{2}\rightarrow 0
\end{equation}
be a short split exact sequence of C*-algebras. Then the sequence
\begin{equation}\label{Nom4}
0\rightarrow (A\widehat\otimes(C_{1}\otimes\mathcal{K}))[M]\rightarrow (A\widehat\otimes(C\otimes\mathcal{K}))[M]\rightarrow (A\widehat \otimes(C_{2}\otimes\mathcal{K}))[M]\rightarrow 0
\end{equation}
is also a short split exact sequence for any locally convex algebra $A$. Since Karoubi-Villamayor algebraic K-functors preserve short split exact exactness of rings, it follows that
the sequence
\begin{multline*}
0\rightarrow kv_{n}((A \widehat\otimes(C_{1} \otimes\mathcal{K}))[M])\rightarrow kv_{n}((A\widehat\otimes(C\otimes\mathcal{K}))[M])\\
\rightarrow kv_{n}((A\widehat\otimes(C_{2}\otimes\mathcal{K}))[M])\rightarrow 0
\end{multline*}
is a short split exact sequence of abelian groups.

Let $D$ be a C*-algebra. Then for any locally convex algebra $A$ the canonical homomorphism $D\rightarrow D\otimes\mathcal{K}$ induces the homomorphism
\begin{equation}\label{Nom5}
kv_{n}((A\widehat\otimes(D\otimes \mathcal{K}))[M])\rightarrow kv_{n}(A\widehat\otimes((D\otimes\mathcal{K})\otimes\mathcal{K})[M]).
\end{equation}
On the other hand one has natural isomorphisms
\begin{align*}
&kv_{n}(A\widehat\otimes((D\otimes\mathcal{K})\otimes\mathcal{K})[M])\approx kv_{n}((A\widehat\otimes (D\otimes M_{2}(\mathcal{K}))[M])\\
&\approx kv_{n}((A\widehat\otimes M_{2}(D\otimes\mathcal{K}))[M])\approx kv_{n}((M_{2}(A\widehat\otimes(D\otimes\mathcal{K}))[M])\\
&\approx kv_{n}(M_{2}((A\widehat\otimes (D\otimes\mathcal{K}))[M]).
\end{align*}
and the composite of the induced homomorphism (\ref{Nom5}) with theses isomorphisms gives us the natural homomorphism
$$
kv_{n}((A\widehat\otimes(D\otimes \mathcal{K}))[M])\rightarrow kv_{n}(M_{2}((A\widehat\otimes (D\otimes\mathcal{K}))[M]).
$$
It is well known \cite{Hi} that for any ring $C$ one has the isomorphism
$$
kv_{1}(C)\rightarrow kv_{1}(M_{2}(C))
$$
implying this isomorphism for all K-functors $kv_{n}$, $n\geq 1$.
Therefore the homomorphism
$$
kv_{n}((A\widehat\otimes(D\otimes\mathcal{K}))[M])\rightarrow kv_{n}(M_{2}((A\widehat\otimes(D\otimes \mathcal{K}))[M])
$$
is an isomorphism for all $n\geq 1$ and we conclude that the induced homomorphism (\ref{Nom5}) is an isomorphism for all $n\geq 1$. Thus by Theorem \ref{th_3.6} the functor $kv_{n}((-\widehat \otimes \mathfrak{K})[M])$, for all $n\geq 1$, is a smooth homotopy functor on the category of locally convex algebras. Taking into account Lemma \ref{lem_3.1} we finally conclude that for any monoid $M$
(satisfying (\ref{Nom1})) the functor $kv_{n}$ is a smooth homotopy functor on the category of monoid algebras over quasi-stable locally convex algebras for all $n\geq 1$.

The exact sequence
$$
0\rightarrow \Omega_{sm}(A)\rightarrow\mathfrak{I}(A)\rightarrow A\rightarrow 0
$$
induces the exact sequence
\begin{equation}\label{Nom2}
0\rightarrow (\Omega_{sm}(A)\widehat\otimes\mathcal{K})[M]\rightarrow (\mathfrak{I}(A)\widehat\otimes\mathcal{K})[M]\rightarrow (A \widehat \otimes\mathcal{K})[M]\rightarrow 0.
\end{equation}
Since $\mathfrak{I}(A)\widehat\otimes\mathcal{K})[M]$ is isomorphic to $\mathfrak{I}(A\widehat\otimes\mathcal{K})[M])$, this sequence is a $GL$-fibration with respect to the smooth cotriple and therefore with respect to the polynomial cotriple. Thus the sequence (\ref{Nom5}) induces the following long exact sequences
\begin{align*}
&\cdots \rightarrow K^{sm}_{n+1}(A\widehat\otimes\mathcal{K})[M])\rightarrow K^{sm}_{n}(\Omega_{sm}(A)\widehat\otimes\mathcal{K})[M])\rightarrow K^{sm}_{n}(\mathfrak{I}(A)\widehat\otimes\mathcal{K})[M])\\
&\rightarrow K^{sm}_{n}(A\widehat\otimes\mathcal{K})[M])\rightarrow K^{sm}_{n-1}(\Omega_{sm}(A)\widehat\otimes \mathcal{K})[M])\rightarrow \cdots,
\end{align*}
\begin{align*}
&\cdots\rightarrow kv_{n+1}(A\widehat\otimes\mathcal{K})[M])\rightarrow kv_{n}(\Omega_{sm}(A)\widehat\otimes \mathcal{K})[M])\rightarrow kv_{n}(\mathfrak{I}(A)\widehat \otimes\mathcal{K})[M])\\
&\rightarrow kv_{n}(A\widehat\otimes\mathcal{K})[M])\rightarrow kv_{n-1}(\Omega_{sm}(A)\widehat\otimes\mathcal{K})[M])\rightarrow \cdots .
\end{align*}
As we have seen the locally convex algebra $\mathfrak{I}(A)$ is smoothly contractible. This imply the contractibility of $(\mathfrak{I}(A) \widehat \otimes\mathcal{K})[M]$. Since $K^{sm}_{n}$ and $kv_{n}$ are smooth homotopy functors, one gets the equalities $K^{sm}_{n}(\mathfrak{I}(A) \widehat \otimes\mathcal{K})[M])= 0$ and $kv_{n}(\mathfrak{I}(A) \widehat \otimes\mathcal{K})[M])= 0$ for all $n\geq 1$.

By using the above long exact sequences one has isomorphisms
\begin{align*}
& K^{sm}_{n}(A\widehat\otimes\mathcal{K})[M])\approx K^{sm}_{1}(\Omega^{n-1}_{sm}(A)\widehat\otimes \mathcal{K})[M]),\\
& kv_{n}(A\widehat\otimes\mathcal{K})[M])\approx kv_{1}(\Omega^{n-1}_{sm}(A)\widehat\otimes\mathcal{K})[M])
\end{align*}
for all $n>1$, where $\Omega^{n}_{sm}(A)= \Omega^{1}_{sm}(\Omega^{n-1}_{sm}(A))$.

By Theorem \ref{th_3.3} the abelian groups $K^{sm}_{1}(\Omega^{n-1}_{sm}(A) \widehat\otimes\mathcal{K})[M])$ and \break $kv_{1}(\Omega^{n-1}_{sm}(A)\widehat \otimes \mathcal{K})[M])$ are isomorphic and finally we obtain the required isomorphism:
$$
K^{sm}_{n}((A\widehat\otimes\mathcal{K})[M])\approx kv_{n}((A \widehat\otimes\mathcal{K})[M])
$$
for all $n\geq 1$. This completes the proof.
\end{proof}

Theorem \ref{th_3.8} generalizes Higson's result \cite{Hi} on the isomorphism of Karoubi-Villamayor algebraic K-functors and topological K-functors for stable $C\star$-algebras, since in this case the smooth K-functors are isomorphic to topological K-functors (see Theorem 1.10, \cite{InKa1})

%% Theorem
\begin{thm}\label{th_3.9}
For any Frechet algebra $A$ with properly uniformly bounded approximate unit there is an
isomorphism
$$
K_{n}((A\widehat\otimes\mathcal{K})[M])\rightarrow K^{sm}_{n}((A\widehat\otimes\mathcal{K})[M])
$$
for all $n\geq 1$.
\end{thm}
\begin{proof}
To prove that $K_{n}((A\widehat\otimes(-\otimes\mathcal{K}))[M])$, $n\geq 1$, is a
smooth homotopy functor on the category of Frechet algebras with
properly uniformly bounded approximate unit we will again use
Theorem \ref{th_3.6} by taking for the functor $T$ the Quillen algebraic K-functor $K_{n}$.
Consider the short exact sequences (\ref{Nom3}) and (\ref{Nom4}). If $A$ is a Frechet algebra  with
properly uniformly bounded approximate unit,
then so it is for the Frechet algebra $A\widehat\otimes(C_{1}\otimes\mathcal{K})$ and consequently by Propostion \ref{prop_3.5} and
Theorem \ref{th_2.9} the monoid algebra $(A\widehat\otimes(C_{1}\otimes\mathcal{K}))[M]$ has the excision property in algebraic K-theory and is H-unital.
Thus the split exact sequence (\ref{Nom3}) induces the split exact sequence
\begin{multline*}
0\rightarrow K_{n}((A \widehat\otimes(C_{1} \otimes\mathcal{K}))[M])\rightarrow K_{n}((A\widehat\otimes(C\otimes\mathcal{K}))[M])\\
\rightarrow K_{n}((A\widehat\otimes(C_{2}\otimes\mathcal{K}))[M])\rightarrow 0.
\end{multline*}
Since $(A\widehat\otimes(D\otimes\mathcal{K}))[M]$ is H-unital for any Frechet algebra $A$ with properly uniformly bounded approximate unit and any C*-algebra $D$, it has the Morita equivalence property implying the isomorphism
$$
K_{n}((A\widehat\otimes(D\otimes \mathcal{K}))[M])\rightarrow K_{n}(M_{2}((A\widehat\otimes (D\otimes\mathcal{K}))[M])
$$
induced by the canonical homomorphism $D\rightarrow D\otimes \mathcal{K}$.Therefore by Theorem \ref{th_3.6} the functor $K_{n}((-\widehat \otimes \mathcal{K})[M])$, for all $n\geq 1$, is a smooth homotopy functor on the category of Frechet algebras with properly uniformly bounded approximate unit. Applying now Lemma \ref{lem_3.1} we show that the Quillen algebraic K-functor $K_{n}$ is a smooth homotopy functor on the category of monoid algebras over quasi-stable Frechet algebras with properly uniformly bounded approximate unit for all $n\geq 1$.

If $A$ is a Frechet algebra with properly uniformly bounded approximate unit, then $\Omega_{sm}(A)\widehat\otimes\mathcal{K}$ is also a Frechet algebra with properly uniformly bounded approximate unit. Thus the algebra $(\Omega_{sm}(A)\widehat\otimes\mathcal{K})[M]$ has the excision property in algebraic K-theory. It follows that the short exact sequence (\ref{Nom2}) induces the following long exact sequence
\begin{align*}
&\cdots \rightarrow K_{n+1}((A\widehat\otimes\mathcal{K})[M])\rightarrow K_{n}(\Omega_{sm}(A)\widehat\otimes\mathcal{K})[M])\rightarrow K_{n}((\mathfrak{I}(A)\widehat\otimes\mathcal{K})[M])\\
&\rightarrow K_{n}((A\widehat\otimes\mathcal{K})[M])\rightarrow K_{n-1}((\Omega_{sm}((A\widehat\otimes\mathcal{K})[M])\rightarrow \cdots .
\end{align*}
Since $K_{n}$ is a smooth homotopy functor, one has the equality \break$K_{n}((\mathfrak{I}(A)\widehat\otimes\mathcal{K})[M])=0$ for $n\geq1$ which implies the isomorphism
$$
K_{n}((A\widehat\otimes\mathcal{K})[M])\approx K_{1}((\Omega^{n-1}_{sm}(A\widehat\otimes\mathcal{K})[M]),n>1.
$$
The group $K_{1}((\Omega^{n-1}_{sm}(A)\widehat\otimes\mathcal{K})[M])$ is isomorphic to the group \break$K^{sm}_{1}((\Omega^{n-1}_{sm}(A)\widehat\otimes\mathcal{K})[M])$
by Theorem \ref{th_3.3} and the composition of these two isomorphisms with the isomorphism
$$
 K^{sm}_{1}((\Omega^{n-1}_{sm}(A)\widehat\otimes\mathcal{K})[M])\approx K^{sm}_{n}((A\widehat\otimes\mathcal{K})[M])
$$
gives us the required isomorphism
$$
K^{sm}_{n}((A\widehat\otimes\mathcal{K})[M])\approx K_{n}((A \widehat\otimes\mathcal{K})[M])
$$
for all $n\geq 1$. This completes the proof.
\end{proof}

%% Theorem
\begin{thm}\label{th_3.10}
If $A$ is a Frechet algebra  with properly uniformly bounded approximate unit, then one has isomorphisms
$$
K_{n}(A\widehat\otimes\mathcal{K})\approx K_{n}((A\widehat\otimes\mathcal{K})[x_{1},x_{2}, \ldots, x_{m}])
$$
for $n,m\geq 1$.
\end{thm}
\begin{proof}
Consider the following commutative diagram
$$
\xymatrix{
 K_{n}(A\widehat\otimes\mathcal{K}) \ar[d]\ar[r] & K_{n}((A\widehat\otimes\mathcal{K})[x_{1}, \ldots, x_{m}])\ar[d]\\
 kv_{n}(A\widehat\otimes\mathcal{K}) \ar[r] & kv_{n}((A\widehat\otimes\mathcal{K})[x_{1}, \ldots, x_{m}])\\
 }
$$
with natural homomorphisms. According to Theorems \ref{th_3.8} and \ref{th_3.9} the Quillen algebraic K-groups $K_{n}((A\widehat\otimes\mathcal{K})[M])$ are isomorphic to Karoubi-Villamayor algebraic K-groups $kv_{n}((A\widehat\otimes\mathcal{K})[M])$ for monoid algebras over quasi stable Frechet algebras with properly uniformly bounded approximate unit for $n\geq 1$. Therefore the vertical homomorphisms are isomorphisms. On the other hand the bottom homomorphism is an isomorphism for any ring. We conclude that the top homomorphism is an isomorphism too. This completes the proof.
\end{proof}

%% Corollary
\begin{cor}
Let $A$ be a Frechet algebra with properly uniformly bounded approximate unit. Then one has isomorphisms
$$
K_{n}((A\widehat\otimes\mathcal{K})[t,t^{-1}]) \approx K_{n}(A\widehat\otimes\mathcal{K})\oplus K_{n-1}(A\widehat\otimes\mathcal{K})
$$
and
$$
K^{sm}_{n}((A\widehat\otimes\mathcal{K})[t,t^{-1}]) \approx K^{sm}_{n}(A\widehat\otimes\mathcal{K})\oplus K^{sm}_{n-1}(A\widehat\otimes\mathcal{K})
$$
for all $n\geq 1$.
\end{cor}
\begin{proof}
Denote $B=A\widehat\otimes\mathcal{K}$ and let $B^{+}$ be the unital Frechet algebra obtained by adding the unit to $B$. The following split exact sequences
$$
0\rightarrow B\rightarrow B^{+}\rightarrow\mathbb C\rightarrow 0,
$$
$$
0\rightarrow B[t]\rightarrow B^{+}[t]\rightarrow\mathbb C[t]\rightarrow 0
$$
and
$$
0\rightarrow B[t,t^{-1}]\rightarrow B^{+}[t,t^{-1}]\rightarrow\mathbb C[t,t^{-1}]\rightarrow 0
$$
 induce the split exact sequences of K-groups
\begin{align*}
&0\rightarrow K_{n}(B)\rightarrow K_{n}(B^{+})\rightarrow K_{n}(C)\rightarrow 0,\\
&0\rightarrow K_{n}(B[t])\rightarrow K_{n}(B^{+}[t])\rightarrow K_{n}(C[t])\rightarrow 0,\\
&0\rightarrow K_{n}(B[t,t^{-1}])\rightarrow K_{n}(B^{+}[t,t^{-1}])\rightarrow K_{n}(C[t,t^{-1}])\rightarrow 0
\end{align*}
for $n\geq 0$, since the algebras $B$, $B[t]$ and $B[t,t^{-1}]$ have the excision property in algebraic K-theory.

If we consider the commutative diagram

$$
\xymatrix{
0\ar[r] & K_{n}(B) \ar[d]\ar[r] & K_{n}(B^{+}) \ar[d]\ar[r] & K_{n}(C) \ar[d]\ar[r] & 0\\
0\ar[r] & K_{n}(B[t])\ar[r] & K_{n}(B^{+}[t])\ar[r] & K_{n}(C[t]) \ar[r] & 0\\
 }
$$
with vertical homomorphisms of K-groups induced by natural injections, we conclude that the middle vertical homomorphism is an isomorphism,
since the right vertical one is a well known isomorphism and the other left vertical homomorphism is an isomorphism too by Theorem \ref{th_3.10}

Thus the Fundamental Theorem of algebraic K-theory gives us the following isomorphisms
\begin{align*}
&K_{n}(B^{+}[t,t^{-1}])\approx K_{n}(B^{+})\oplus K_{n-1}(B^{+}),\\
&K_{n}(\mathbb C[t,t^{-1}])\approx K_{n}(\mathbb C)\oplus K_{n-1}(\mathbb C)
\end{align*}
for $n\geq 1$. From these isomorphisms and the above obtained short split exact sequences of $K$-groups immediately follows the required first isomorphism of the Corollary.The second isomorphism is a consequence of the first isomorphism and Theorem \ref{th_3.9}.This completes the proof.
\end{proof}

%% Remark
\begin{rem}
I would like to thank Guillermo Cortinas informing me that for countable monoid $M$ Theorems 3.8 and 3.9 can be obtained respectively by Theorem 6.2.1 \cite{CoTh} and for m-convex Frechet algebra with uniformly bounded approximate unit by applying argument of Theorem 12.1.1 and Remark 12.1.4 \cite{Co}.
\end{rem}

%%%%%%%%%

\

\


\begin{bibdiv}
\begin{biblist}


\bib{Ba}{article}{
author={Bass, H.},
title={Some problems in classical algebraic K-theory, Algebraic K-Theory II},
booktitle={Lecture Notes in Math. 342, Springer-Verlag, New York, Berlin},
date={1973},
pages={1-73},
}

\bib{Co}{article}{
author={Cortinas, G},
title={Algebraic v. topological K-theory: a friendly match. In: Topics in algebraic and topological K-theory},
journal={Springer Lecture Notes in Mathematics},
date={2008},
}


\bib{CoTh}{article}{
author={Cortinas, G.},
author={Thom, A.},
title={Comparison between algebraic and topological K-theory of locally convex algebras},
journal={Adv. Math.},
volume={218},
date={2008},
pages={266-307},
}

\bib{CoHaWe}{article}{
author={Cortinas, G.},
author={Haesemeyer, C.},
author={Weibel, C.},
title={cdh-fibrant Hochschild homology and a conjecture of Vorst},
journal={ J. Amer. Math. Soc.},
volume={21},
date={2008},
pages={547-561},
}

\bib{Ge}{article}{
author={Gersten, S.},
title={Higher K-theory of rings, Algebraic K-theory I},
booktitle={Lecture Notes in Math. 341, Springer-Verlag, New York, Berlin},
date={1973},
pages={1-40},
}


\bib{Gu}{article}{
author={Gubeladze, J.},
title={Classical algebraic K-theory of monoid algebras},
booktitle={Lecture Notes in Math. 1437, Springer-Verlag, New York, Berlin},
date={1990},
pages={36-94},
}


\bib{Gu1}{article}{
author={Gubeladze, J.},
title={K-theory of affine toric varieties},
journal={Homology,Homotopy and Applications},
volume={1(5)},
date={1999},
pages={135-145},
}


\bib{Hi}{article}{
author={Higson, N.},
title={Algebraic K-theory of stable C*-algebras},
journal={Adv. Math.},
volume={67},
date={1988},
pages={1-140},
}


\bib{In}{article}{
author={Inassaridze, H.},
title={Algebraic K-theory of normed algebras},
journal={K-Theory},
volume={21(1)},
date={2000},
pages={25-56},
}


\bib{InKa}{article}{
author={Inassaridze, H.},
author={Kandelaki, T.},
title={K-theory of stable generalized operator algebras},
journal={K-Theory},
volume={27},
date={2002},
pages={103-110},
}


\bib{InKa1}{article}{
author={Inassaridze, H.},
author={Kandelaki, T.},
title={Smooth K-theory of locally convex algebras},
journal={Communications in Contemp. Math.},
volume={13(4)},
date={2011},
pages={553-577},
}


\bib{Ka}{article}{
author={Karoubi, M.},
title={La periodicite de Bott en K-theorie generale},
journal={Ann. Sci. Ec. Norm. Sup.},
volume={4},
date={1971},
pages={63-95},
}


\bib{KaVi}{article}{
author={Karoubi, M.},
author={Villamayor, O.},
title={K-theorie algebrique et K-theorie topologique},
journal={J. Math. Scand.},
volume={28(2)},
date={1971},
pages={265-307},
}


\bib{Qu}{article}{
author={Quillen, Q.},
title={Higher algebraic K-theory I, Algebraic K-theory I},
booktitle={Lecture Notes in Math. 341, Springer-Verlag, New York, Berlin},
date={1973},
pages={77-139},
}


\bib{Ro}{article}{
author={Rosenberg, J.},
title={The algebraic K-theory of operator algebras},
journal={K-Theory},
volume={12(1)},
date={1997},
pages={75-99},
}


\bib{Ro1}{article}{
author={Rosenberg, J.},
title={Comparison between algebraic and topological K-theory for Banach algebras and C*-algebras},
booktitle={in Handbook of K-theory, eds. E.M.Friiedlander and D.Grayson, Springer-Verlag},
date={2004},
pages={843-874},
}


\bib{SuWo}{article}{
author={Suslin, A.},
author={Wodzicki, M.},
title={Excision in algebraic K-theory},
journal={Ann. Math.},
volume={136(1)},
date={1992},
pages={51-122},
}


\bib{Sw}{article}{
author={R.G.Swan, R.G.},
title={Some relations between higher K-functors},
journal={J. Algebra},
volume={21(1)},
date={1972},
pages={113-136},
}



\bib{Ti}{article}{
author={Tillmann, U.},
title={K-theory of fine topological algebras, Chern character and assembly},
journal={K-Theory},
volume={6(1)},
date={1992},
pages={57-86},
}




\bib{Tr}{book}{
author={Treves, F.},
title={Topological Vector Spaces, Distributions and Kernels},
series={}
publisher={Academic Press},
place={New York, London},
date={1967}
}


\bib{Wo}{article}{
author={Wodzicki, M.},
title={Excision in cyclic homology and in rational algebraic K-theory},
journal={Ann. Math.},
volume={120},
date={1989},
pages={591-639},
}


\bib{Wo1}{article}{
author={Wodzicki, M.},
title={Algebraic K-theory and functional analysis},
booktitle={in First European Congress of Mathematics, Vol. II (Paris,1992), Progr. 120, Birkhauser, Basel},
date={1994},
pages={485-496},
}

\end{biblist}
\end{bibdiv}
\end{document}